\documentclass[12pt]{amsart}
\usepackage{amsrefs}

\usepackage{amsmath,amsthm,amsopn,amsfonts,amssymb,color}
\usepackage[colorlinks=true,linkcolor=blue,urlcolor=blue]{hyperref} 
\usepackage{cancel}
 \usepackage{tikz}
\usetikzlibrary{intersections,calc}
\usepackage{graphicx}
\usepackage{float}
\usepackage{amsmath}
\usepackage{fancyhdr}
\usepackage{epstopdf}
\usepackage{amsfonts}
\usepackage{pgfplots}
\pgfplotsset{compat=newest}
\usetikzlibrary{intersections, pgfplots.fillbetween}
\pgfdeclarelayer{pre main}
\usepackage{amsmath}
\DeclareMathOperator\arctanh{arctanh}
\usepackage{booktabs}
\usepackage{cancel}
\usepackage{times}
\usepackage{mathrsfs}
\usepackage{multirow,multicol}
\usepackage{amsfonts}
\usepackage{amsthm}
\usepackage{amsmath}
\usepackage{amsmath,amssymb}
\newtheorem{teo}{Theorem}[section]
\newtheorem{prop}[teo]{Proposition}

\newtheorem{remark}[teo]{Remark}

\newtheorem{lemma}[teo]{Lemma}

\newcommand\R{\mathbb{R}}

\def\dun{\nabla{u_{n}}}

\def\ipr{\int_{\{ |E|\leq\rho\}}}
\def\disp{\displaystyle}

\def\elle#1{L^{#1}(\Omega)}

\def\div{{\rm div}}
\def\elle#1{L^{#1}(\Omega)}
\def\w#1#2{W^{#1,#2}_0(\Omega)}
\def\io{\int_{\Omega}}
\def\norma#1#2{\|#1\|_{\lower 4pt \hbox{$\scriptstyle #2$}}}
\def\un{u_n}

\def\finedim
\def\gw{G_{\tilde{k}}(w_n)}
\def\wn{w_n}

\def\R{I \!\!R}

\def\elle#1{L^{#1}(\Omega)}

\def\be{\begin{equation}}
\def\ee{\end{equation}}

\usepackage{todonotes}

 \usepackage{color}
 \numberwithin{equation}{section}

\title[Elliptic problems]{Elliptic problems with superlinear convection terms}
\author{Lucio Boccardo}
\address[L.\ Boccardo]{Sapienza University of Rome, Piazzale Aldo Moro 5, 00185 Rome, Italy}
\email{boccardo@mat.uniroma1.it}

\author{Stefano Buccheri}
\address[S.\ Buccheri]{University of Vienna,
Oskar-Morgenstern-Platz 1, 1090 Vienna,
Austria}
\email{stefano.buccheri@univie.ac.at}

\author{G. Rita Cirmi}
\address[R. Cirmi]{University of Catania, Viale Adrea Doria 6, 95125 Catania, Italy}
\email{cirmi@dmi.unict.it}

\date{\today}
\keywords{superlinear convection terms, non coercive operators}

\begin{document}

\maketitle 
 \begin{abstract}
 In this manuscript we deal with elliptic equations with superlinear first order terms in divergence form of the following type
\[
-\div(M(x)\nabla u)= -\div(h(u)E(x))+f(x),
\]
where $M$ is a bounded elliptic matrix, the vector field $E$ and the function $f$ belong to suitable Lebesgue spaces, and the function $s\to h(s)$ features a superlinear growth at infinity.
We provide some existence and non existence results for solutions to the associated Dirichlet problem and a comparison principle. 
\end{abstract}

\tableofcontents

\section{Introduction and main results}
The aim of the present work is to provide some existence and uniqueness results for solutions to the following nonlinear problem
\begin{equation} \label{problem}
\begin{cases}
 -\div(M(x)\nabla u)= -\div(h(u)E(x))+f(x) \qquad & \mbox{in } \Omega,\\
u (x) = 0 & \mbox{on }  \partial \Omega.
\end{cases}
\end{equation}
Throughout the paper, we assume that $\Omega$ is a bounded, open subset of $\mathbb{R}^N$ with $N > 2$, $M(x)$ is a measurable matrix such that
\begin{equation}\label{alfa}
 \alpha |\xi|^2\leq 
  M(x)\xi\cdot\xi \leq\beta|\xi|^2,  \quad  \mbox{a.e.}\; x\in\Omega,\quad \forall\; \;\xi\in\R^N,
\end{equation}
with $\alpha, \beta>0$.  The vector field $E(x)$ and the function $f(x)$ belong to suitable Lebesgue spaces and the real function $s\to h(s)$ satisfies
\begin{equation}
\label{ipoh}
h\in W^{1,\infty}_{loc}(\mathbb{R}),  \ \ \  h(0)=0, \ \ \ \mbox{and} \ \ \ \lim_{|s| \to  \infty} \frac{|h(s)|}{|s|}=+\infty.
\end{equation}
The first two assumptions on $h$ are needed for the validity of the comparison principle and for $0$ to be the solution of the homogeneous case ($f\equiv0$).
The main focus here is the behavior at infinity of the nonlinearity $h(s)$, in particular, its {\it superlinear} growth.
Examples of such $h(s)$ are 
\begin{equation}
\label{slow}
h(s)= s \log^{\theta} (e+|s|), \quad \forall s \in \mathbb{R},
\end{equation}
and
\begin{equation}
\label{strong}
h(s)= s |s|^{\theta}  \ \ \ \mbox{or} \ \ \ h(s)=  |s|^{1+\theta} \quad \forall s \in \mathbb{R},
\end{equation}
with $\theta >0$. Let us recall that a function $u \in W^{1,2}_0(\Omega)$ is a weak solution to problem \eqref{problem} if $f \in L^{\frac{2N}{N+2}}(\Omega)$, $h(u)E \in [L^2(\Omega)]^N$, and the following identity holds
\begin{equation}
\label{weakformulation}
\int_{\Omega} M(x) \nabla u \nabla \varphi =\int_{\Omega}h(u)E(x) \nabla \varphi  + \int_{\Omega} f(x) \varphi \quad \forall \varphi \in  W^{1,2}_0(\Omega).
\end{equation}
The main difficulty in dealing with problem \eqref{problem} is the loss of coercivity of the nonlinear differential operator $A(u)= -\mbox{div}\big(M(x)\nabla u-h(u)E(x)\big)$. This phenomenon is already present in the linear case $h(s)=s$ and it is even stronger in our framework. For the sake of clarity take 
\[
h(s)=s\log(e+|s|), \quad \Omega=B_1(0), \quad E=-x, \quad \mbox{and} \quad v=1-|x|^2. 
\]
A slightly formal computation shows that  
\[
\langle A(tv),tv \rangle\le t^2 \beta\io |x|^2- t^2\io (1-|x|^2)\log(e+t(1-|x|^2))|x|^2\to\, -\infty \quad \mbox{as } t\to+\infty.
\]
Therefore, the drift term appearing in \eqref{problem} has to be considered as a \emph{reaction term} and it may represent an obstruction for the existence of solutions. Let us also explicitly point out that we will never take advance of some cancellation properties of the vector field $E$ or some sign conditions when searching for a priori estimates.\\

Problem \eqref{problem} with $h(s)= s$ is classical and has already been addressed by many authors. We defer to the next section the general analysis of the bibliographic background and focus on the strategy introduced in \cite{Bumi2009}. The main tool of that paper is the following a priori decay for the measure of superlevel sets of solutions
\be\label{logest}
|\{u> k\}|^{\frac{1}{2^*}}\le \frac{C}{[\log(1+k)]}\left[\io |E|^2 +\io |f|\right]^{\frac12},
\ee
which is obtained testing \eqref{weakformulation} (in the case $h(s)=s$) with
\[
\varphi(u)=\frac{u}{|u|+1}=\int_0^u\frac{ds}{(|s|+1)^2}.
\]
Combining \eqref{logest} with a truncation argument, it is possible to bypass the loss of coercivity due to $-$div$(uE)$ and obtain a priori estimates for $u$ in $W^{1,2}_0(\Omega)$.\\
It is therefore natural to wonder if we can adapt such technique for a general nonlinearity $h(s)$. A reasonable choice for $\varphi(u)$ is
\be\label{nonlintest}
\varphi(u)=\int_0^u\frac{ds}{(|h(s)|+1)^2},
\ee
and, indeed, we formally get that 
\be \label{decayintro}
|\{|u|>|k|\}|^{\frac{2}{2^*}}\le \frac{C}{H(k)^2}\left[\io |E|^2 +\io |f|\right],
\ee
where 
\begin{equation}
\label{defH}
H(s)=\int_0^s\frac{dt}{|h(t)|+1}
\end{equation}
(see Lemma \ref{decay} below for the rigorous statement and proof). Obviously, formula \eqref{decayintro} provides a decay estimate only if
\be\label{growthg}
\lim_{|s|\to \infty}|H(s)|=\infty.
\ee
Simple computation shows that such a property is satisfied for $h(s)$ as in \eqref{slow} while is not in case \eqref{strong}.
Our first result shows that \eqref{growthg} is a sufficient condition to guarantee existence of (bounded) solutions to problem \eqref{problem}.

\begin{teo}\label{generalh}
Assume \eqref{alfa}, \eqref{ipoh}, \eqref{growthg} and take $E\in [\elle r]^N$,  with $ r>N$, and $f\in \elle m$, with $ m>\frac{N}{2}$. Then, there exists a unique weak solution $u\in \w12\cap \elle{\infty}$ to problem \eqref{problem}.
\end{teo}
The interest of such a result is that, despite a slight but still superlinear growth for the function $h(s)$, problem \eqref{problem} is well posed without any further assumption on the size of $E$ or $f$. Let us notice that if $f\ge 0$, Theorem \ref{generalh} is valid if $\lim_{s\to \infty} |H(s)|=+\infty$. Similarly, if $f\le0$, it is enough to assume $\lim_{s\to -\infty} |H(s)|=+\infty$. This follows easily by the comparison principle given in Theorem \ref{comparison} below.

 In the special case \eqref{slow} we are also able to provide existence of unbounded solutions, as the following Theorem states.
\begin{teo}
\label{logarithmich}
Assume \eqref{alfa}, take
$E\in [\elle r]^N$,  with $ r>N$, and 
$f\in  \elle {2_{*}}$, with $2_{*}=\frac{2N}{N+2}$. Then, there exists a unique weak solution to
\begin{equation} \label{logproblem}
\begin{cases}
 -\div(M(x)\nabla u)= -\div(u\log (e+|u|)E(x))+f(x) \qquad & \mbox{in } \Omega,\\
u (x) = 0 & \mbox{on }  \partial \Omega.
\end{cases}
\end{equation}
\end{teo}

Let us stress again that main point of Theorems \ref{generalh} and \ref{logproblem} is the following: if the function $s\to h(s)$ does not grow to much at infinity (a slightly superlinear growth is allowed), problem \eqref{problem} is uniquely solvable \emph{for all} $E$ and $f$ in some suitable Lebesgue spaces, without any further condition on the size of their norms (or on the sign of the lower order term).\\

If we assume $ h(s)=s|s|^{\theta} $, with $\theta>0$, condition \eqref{growthg} is not satisfied anymore, and to prove existence of solution we need a control on the size of the data, as the following Theorem states. In the following $\mathcal{S}$ is the Sobolev constant of the embedding $W^{1,2}(\Omega)$ in $L^{2^*}(\Omega)$ and we use the notation $m^{**}=\frac{mN}{N-2m}$.
\begin{teo}\label{fixedpoint}
Assume \eqref{alfa}, take $\theta>0$, $m\in [2N/(N+2),N/2)$, and assume moreover that
\begin{equation}\label{27-10bis}
0< \frac1N-\frac1r=\frac{\theta}{m^{**}}, \ \ \ f\in\elle m \ \ \ \mbox{and} \ \ \ E\in (\elle r)^N.
\end{equation}
If
\begin{equation}\label{smallass}
\|f\|_{\elle m}\|E\|_{\elle r}^{\frac{1}{\theta}}\le \frac{\theta}{\mathcal{S}} \left(\frac{\alpha2^*}{\mathcal{S}m^{**}(1+\theta)}\right)^{1+\frac1{\theta}},
\end{equation}
then there exists a unique $u\in W^{1,2}_0(\Omega)\cap\elle{m^{**}}$ solution of 
\begin{equation} \label{problembis}
\begin{cases}
-\div(M(x)\nabla u)= -\div(u|u|^{\theta}E(x))+f(x) \qquad & \mbox{in } \Omega,\\
u (x) = 0 & \mbox{on }  \partial \Omega.
\end{cases}
\end{equation}
\end{teo}

The proof of Theorem \ref{fixedpoint} relies on the use of Schauder fixed point theorem and its more delicate step is finding a suitable invariant set for the map 
\[
\elle{m^{**}} \ni v\mapsto S(v)=w \in \elle{m^{**}},
\]
where $w$ is the unique weak solution of  $$-\div(M(x)\nabla w)= -\div(v|v|^{\theta}E(x))+f(x).$$
It is worthy to mention that we stated Theorem \eqref{fixedpoint} with $h(s)=s|s|^{\theta}$ for simplicity. An equivalent result can be obtained with $h(s)=|s|^{1+\theta}$ or for a more general $h(s)$ with a suitable power-type control on its growth.\\

To show that \eqref{growthg} and \eqref{smallass} are somehow connected and that they are natural assumptions when searching for solutions, we now provide a sharp result for positive solutions in the radially symmetric case.
\begin{prop}\label{radialcase}
Let 
\[
E(|x|)=-K\frac{x}{|x|}, \quad f(|x|)=\frac{{\epsilon}}{N-1} \frac1{|x|}, \quad \Omega=B_R,
\]
for some $K,{\epsilon},R>0$. Then, the problem 
\begin{equation} \label{onedproblem}
\begin{cases}
-\Delta  u=- \div(h(u) E(x))+f(x) \qquad & \text{in } B_R(0),\\
u (x) = 0 & \text{on }  \partial B_R(0),
\end{cases}
\end{equation}
admits a positive solution $u\in W^{1,2}_0(\Omega)\cap \elle{\infty}$ if and only if \begin{equation}\label{oss}
\int_0^{\infty}\frac{ds}{Kh(s)+{\epsilon}}>R.
\end{equation}
\end{prop}
Proposition \ref{radialcase} shows that \eqref{growthg} is sharp in order to have existence of bounded solutions regardless the size of the data. Moreover it provides a link between \eqref{growthg} and  \eqref{smallass} when $h(s)=s^{1+\theta}$ (here we deal with positive solutions). Indeed, for this concrete choice of $h$, a simple change of variable shows that \eqref{oss} is equivalent to 
\[
I_{\theta}=\int_0^{\infty}\frac{ds}{s^{1+\theta}+1}>K^{\frac{1}{1+\theta}}{\epsilon}^{\frac{\theta}{1+\theta}}R.
\]
Taking $m,r,\theta$ as in \eqref{27-10bis} and writing the right hand side above as a function of $\|f\|_{\elle m}$ and $\|E\|_{\elle r}$ (for the radially symmetric choice of $f$ and $E$ of Proposition \ref{radialcase}), it follows that
\[
\|f\|_{\elle m}\|E\|_{\elle r}^{\frac{1}{\theta}}<cI_{\theta}^{\frac{1+\theta}{\theta}},
\]
and we see again the same structure of assumption \eqref{smallass}.\\

Condition \eqref{oss} (and the proof of Proposition \ref{radialcase}) is clearly related to the following one-dimensional initial value problem
\begin{equation}\label{onedim}
\begin{cases}
 -\bar u'(r)=K h( \bar u(r))+{\epsilon} \qquad & \mbox{in } (0,R),\\
u (R) = 0. & 
\end{cases}
\end{equation}
Notice that, if \eqref{oss} does not hold, the problem above has a blow up for some point in $[0,R)$.\\
Quite interestingly, the test functions based approach outlined before Theorem \ref{generalh} echoes the procedure of solving \eqref{onedim} by separation of variables, see Remarks \ref{explicit1} and \ref{explicit2} .\\

Finally, let us focus on the problem 
\begin{equation} \label{zeroorderterm}
\begin{cases}
 -\div(M(x)\nabla u)+\mu\, u= -\div(u|u|^{\theta}E(x))+f(x) \qquad & \mbox{in } \Omega,\\
u (x) = 0 & \mbox{on }  \partial \Omega.
\end{cases}
\end{equation}
Beyond the interest in studying \eqref{zeroorderterm} due to its connection to the evolutionary problem, here we want to highlight the gain in the growth at infinity of the nonlinearity  $h(s)$ allowed by the presence of the zeroth order term $\mu u$. Indeed,  we have the following result.
\begin{teo}\label{+uteo}
Assume \eqref{alfa}, take $\mu>0$, $\theta<1/N$,  $E\in [\elle r]^N$   with  $r=\frac{N}{1-\theta N}$ and $f\in \elle {2_*}$. Then, there exists a unique $u\in \w12$ weak solution of problem \eqref{zeroorderterm}.
\end{teo}
In other words, the operator $u\to -\div(M(x)\nabla u)+\mu u$ allows for a power growth below $1+\frac1N$. We do not know if Theorem \ref{+uteo} holds when the threshold valued is reached, but a simple computation in the radial case shows that problem \eqref{zeroorderterm} is not well posed for general $f$ and $E$ if $\theta>\frac1N$ (see Remark \ref{lastrem}).

The key point in the proof of the previous result is that the presence of the zeroth order term implies an a priori $\elle 1$-estimate for solution to \eqref{zeroorderterm} of the following type  
\[
\io |u|\leq\frac{\norma{f}{1}}{\mu},
\]
see Lemma \ref{elle1estimate} below. It is worthy to notice that such estimate does not depend on the ellipticity constant of the matrix $M(x)$.

\noindent
Our last result takes advantage of such $\elle1$-estimate to provide a $\log$-scale improvement of Theorem \ref{logarithmich}.
\begin{teo}\label{logsummE}
Let the assumption \eqref{alfa} be satisfied,   $E\in [L^N\log^N(\Omega)]^N$ and $f\in \elle {2_*}$. Then there exists a unique solution of the problem 
\begin{equation}
\label{problemalog}
\begin{cases}
 -\div(M(x)\nabla u)+\mu\, u= -\div( u\log(e+|u|)E(x))+f(x) \qquad & \mbox{in } \Omega,\\
u (x) = 0 & \mbox{on }  \partial \Omega.
\end{cases}
\end{equation}
\end{teo}

\subsection{Bibliographic background}
From a broader perspective, problem \eqref{problem} can be seen as the stationary counterpart of the following nonlinear continuity equation
\be\label{evol}
\partial_t u +\mbox{div}(J)=f, \ \ \ J=-M(x)\nabla u+h(u)E(x),
\ee
where the flux $J$ is made up of the diffusion term $-M(x)\nabla u$ and the (possibly) nonlinear convection term $h(u)E(x)$. Such type of equation can be derived as hydrodynamic limit of a family of scaled kinetic equations \cite{degond}, is widely used in biology \cite{hillen} and pedestrians dynamics \cite{iuorio}, and naturally arises in the mean field games theory (see for instance \cite{porrettacon,porretta}). Clearly, if we take $h(u)\equiv u$, we recover the celebrated Fokker-Planck equation. We also mention \cite{gomezcastro} for a recent review on the related class of aggregation-diffusion equations and \cite{cornalba} for a glimpse in the stochastic framework.\\ 

Let us now go back to our stationary problem. Notice that the coefficients of our equation satisfy mild regularity assumptions: the  matrix $M(x)$ is measurable and positive definite and the vector field $E(x)$ simply belongs to $\elle r$. It doesn't have any sort of differentiability nor sign condition that may help for a priori estimates. In particular we do not assume that $E=\nabla V$, for some potential $V$, so that no variational structure is available for us (in such a case \eqref{evol} may be reformulated as a Wasserstein-type gradient flow).\\

Form our perspective the main feature of the operator 
\[
u\to  -\div(M(x)\nabla u)+\div(h(u)E(x))
\]
is the loss of coercivity due to the lower order term. For $h$ with linear growth, different approaches have been developed to overcome such a difficulties. The authors of \cite{betta,delvecchio} apply a symmetrization procedure to deduce an explicit pointwise bound for the decreasing rearrangement of the solution. As we already mentioned, in \cite{Bumi2009,Bumi2012} a log-type estimate allows to deduce a universal decay on the super level set of the solution. This decay, combined with a truncation argument, provides a priory estimates in $\w12$. We also mention \cite{bbc,goffi} for estimates obtained by duality method, \cite{bucc,cirmi2020,cirmi2022,feng,sake} for regularity results in different frameworks, and \cite{jde,farroni2021,farronibis} for equation with coefficients in Marcinkiewicz spaces.\\

 Our work follows this line of investigation with the novelty of a convection term with superlinear growth at infinity. Up to our knowledge, very little has been done in this direction (see the very recent \cite{marah}) and the slightly superlinear growth condition \eqref{growthg} has never been used in this context. The legitimacy of such assumption is provided by Proposition \ref{radialcase} and it represents a nice bridge between the well posedness of problem \eqref{problem} and the blow-up in finite time for the ODE \eqref{onedim}. It would be interesting to understand if Theorem \ref{logarithmich} remains true for a general $h(u)$ that satisfies \eqref{growthg}. 

Also the {strongly superlinear} case $h(s)=s|s|^{\theta}$, treated in Theorems \ref{fixedpoint} and \ref{problembis}, does not seem to be addressed in the literature so far. We believe that it may be of some interest to compare problems \eqref{problembis} and \eqref{zeroorderterm} with the following one
\begin{equation} \label{dquadro}
\begin{cases}
 -\div(M(x)\nabla u)+\mu u= |\nabla u|^q+f(x) \qquad & \mbox{in } \Omega,\\
u (x) = 0 & \mbox{on }  \partial \Omega.
\end{cases}
\end{equation}
with $q\in(1,2]$ and $f\ge0$. The literature on \eqref{dquadro} is huge, see for instance \cite{tom, betta1, bmp84, bmp, grenon, goffi, phuc} and reference therein. One of the most interesting fact about this problem is the nontrivial structure of the set of solutions with respect to the parameter $\mu$ (see \cite{david} and the more recent papers \cite{salva}, \cite{antonio}). Let us recall that, if $\mu=0$, problem \eqref{dquadro} is solvable in $\w12$ if and only if $f$ satisfies some smallness condition, and this is in a kind of analogy with our Problem \ref{problembis}. However, if $\mu >0$, problem \eqref{dquadro} enjoys (for $f\in \elle m$, $m>N/2$) a universal $L^{\infty}$-bound on solutions that is missing in our case when $\theta> \frac1N$, as shown in Remark \ref{lastrem}. This is somehow surprising since a term like $|\nabla u|^2$ seems at first glace much heavier than one like $-$div$(u|u|E)$.\\

We close the introduction with a couple of further comments on \eqref{evol}. Indeed, the next natural step would be to consider the evolutionary counterpart of \eqref{problem} and to address wellpodedness, local/global existence, blow-up (if any), and decay properties of the solution. In the case $h(u)=u$ such issues are treated for instance in \cite{bop}, \cite{bopbis} and \cite{farronibis}. For the special choice $h(u)=u|u|^{\theta}$ a preliminary study is contained in [].

\section{Proof of the results}

Let us start with some notation and  preliminary material that we shall use in what follows. Many of our results are based on the approximation of \eqref{problem} with problems with \emph{trouncated} right hand side (Theorem \ref{fixedpoint} will be proved with a fixed point argument instead). To be more explicit, let us recall the following definitions
\[
T_k(s)=\max\{-k,\min\{s,k\}\}, \ \ \ G_k(s)=s-T_k(s), \quad \mbox{ for } k>0.
\]
Using   Schauder fixed point theorem, we can prove that, for any $n\in\mathbb{N}$,  there exists $\un \in \w12$ which solves
\be\label{approx}
\io M(x)\nabla \un\nabla v 
+\io \mu\un v
=\io h_n(u_n) E_n(x)\nabla v+\io f_n\,v , \quad \forall\; v \in \w12,
\ee
where $h_n(s)= T_n(h(s))$, $f_n= T_n(f)$, and, with a slight abuse of notation, $E_n$ is the vector field with components $(E_n)^i=T_n((E)^i)$, with $i=1,\cdots,N$.\\
Moreover, by the classical Stampacchia's regularity result (see \cite{stamp}), we have that  $\un \in \elle{\infty}$  (notice that for any fixed $n$ the term $ h_n(\un) E_n(x)$ is bounded).\\

The main difficulty here is to achieve a priori estimates for the sequence $\{\un\}$ in some suitable space, and this is done following different strategies depending on the choice of $h(s)$ and on the presence of the zeroth lower order term. Once these estimates are obtained, the passage to the limit in \eqref{approx} always follows the same procedure, that we provide in the following Lemma.
\begin{lemma}\label{passtothelim}
Assume \eqref{alfa}-\eqref{ipoh}, $\mu\ge 0$, that $E\in [\elle 2]^N$, $f\in \elle{\frac{2N}{N+2}}$, and that 
\[
\|\un\|_{\w12}+\|h(\un) E\|_{\elle 2}\le C_1.
\] 
Theerefore, up to a subsequence, $\un\rightharpoonup u$ in $\w12$ and $u$ solves
\[
\io M(x)\nabla u \nabla v 
+\io \mu u v
=\io h(u) E(x)\nabla v+\io fv  \quad \forall\; v \in \w12.
\]
\end{lemma}
\begin{proof}
Since $\{\un\}$ is bounded in $\w12$, we deduce that, up to a subsequence, $\un\to u$ weakly in $\w12$, strongly in $\elle{q}$ with $q<2^*$, and $a.e$ in $\Omega$. This is enough to pass to the limit in the first, in the second, and in the fourth integral of \eqref{approx}. To deal with the third one, notice that, for any measurable $\Sigma\subset\Omega$, we have that
\[
\int_\Sigma
|h_n(\un)E_n(x)\nabla v|
\leq
\bigg(\io|h(\un)E|^2\bigg)^\frac{1}{2}
\bigg(\int_{\Sigma}|\nabla v|^2\bigg)^\frac12\le C_1 \bigg(\int_{\Sigma}|\nabla v|^2\bigg)^\frac12.
\]
This amounts to say that, for any given $v\in W^{1,2}_0(\Omega)$, the term $\{h_n(\un)E_n\nabla v\}$ is equi-integrable. Since $\un\to u$ converges pointwise, we can use Vitali Theorem to conclude that also the third term in \eqref{approx} passes to the limit as $n\to\infty$.
\end{proof}

The next result provides us a comparison principle.

\begin{teo}\label{comparison}
Under conditions \eqref{alfa} and \eqref{ipoh}, let us consider $\mu\ge0$, $f\in\elle{\frac{2N}{N+2}}$, $E\in [\elle 2]^N$, and two functions $v,w\in \w12\cap \elle{\infty}$ such that 
\[
\io M(x)\nabla v \nabla \phi 
+\io \mu v \phi
\le \io h(v) E(x)\nabla \phi+\io f\phi  \quad \forall\; v \in \w12,
\]
and
\[
\io M(x)\nabla w \nabla \phi 
+\io \mu w \phi
\ge \io h(w) E(x)\nabla \phi+\io f\phi  \quad \forall\; v \in \w12,
\]
for any $0\le \phi\in\w12\cap \elle{\infty}$. Therefore $v\le w$ $a.e.$ in $\Omega$.\\
Moreover, if $v,w$ are in $\w12$ (possibly unbounded) and we additionally assume that, for $\theta>0$,
\[
|h'(s)|\le c(|s|^{\theta}+1)\, a.e. \mbox{ in } \mathbb R \quad \mbox{ and } \quad |v|^{\theta}|E|+|w|^{\theta}|E|\in\elle 2,
\]
we can again conclude that $v\le w$ $a.e.$ in $\Omega$.\\
\end{teo}
\begin{proof}
We shall prove that $(v-w)_+\equiv 0$ a.e. in $\Omega$. To this aim, let us use $T_{{\epsilon}}(v-w)_+$ as a test function in the inequalities satisfied by $v$ and $w$. Taking the difference, we obtain
\[
\begin{split}
\alpha\io |\nabla T_{{\epsilon}}(v-w)_+ |^2\le& \int_{0<v-w<{\epsilon}} \big(h(v)-h(w)\big) E(x)\nabla T_{{\epsilon}}(v-w)_+\\
\le& \frac{1}{2\alpha}\int_{\{0< v-w< {\epsilon}\}}\big|h(v)-h(w)\big|^2 |E(x)|^2+\frac{\alpha}2\io |\nabla T_{{\epsilon}}(v-w)_+ |^2.
\end{split}
\]
Setting $z:=(v-w)_+$ and
using Poincar\'e's inequality with its best constant $\lambda_1$ (as in Theorem 6.1 of \cite{Bumi2009}), we deduce that
\[
\lambda_1{\epsilon}^2|\{z>k\}|  \le \lambda_1 \io|T_{{\epsilon}}(z) |^2\le C\int_{\{0< z< {\epsilon}\}}\big|h(v)-h(w)\big|^2|E(x)|^2
\]
for any $k\ge{\epsilon}$. In other words
\be\label{13:35}
|\{z>k\}|\le \frac{C}{\epsilon^2}\int_{\{0< z< {\epsilon}\}}\big|h(v)-h(w)\big|^2|E(x)|^2 \quad \mbox{for any }   {\epsilon}\le k.
\ee
If $v,w\in\elle{\infty}$, let us set $|v|+|w|\le M$ and use the assumption $h\in W^{1,\infty}_{loc}(\mathbb{R})$ to deduce that
\[
|\{z>k\}|\le C L_M^2 \int_{\{0< z< {\epsilon}\}}|E(x)|^2 \quad \mbox{for any }   {\epsilon}\le k,
\]
for a suitable $L_M$ independent on $\epsilon$. Since the measure of the set $\{0< z< {\epsilon}\}$ goes to zero as ${\epsilon}\to0$, we infer that
\be\label{17:50}
|\{z>k\}|=0 \quad  \mbox{for any }    k\ge0,
\ee
and we are done.\\
On the other hand, if $|h'(s)|\le c(|s|^{\theta}+1)$, we have that
\[
|h(v)-h(w)|\le c|v-w|\int_0^1\big(|w+t(v-w)|^{\theta}+1\big)dt\le \tilde c|v-w|(|v|^{\theta}+|w|^{\theta}+1 ),
\]
and \eqref{13:35} becomes
\[
|\{z>k\}|\le C\int_{\{0< z< {\epsilon}\}}(|v|^{2\theta}+|w|^{2\theta}+1 )|E(x)|^2 \quad \mbox{for any }   {\epsilon}\le k.
\]
Therefore, thanks to the assumptions and the absolute continuity of the integral, we can take again the limit as $\epsilon\to0$ and obtain again \eqref{17:50}.
\end{proof}

\subsection{Slightly superlinear case}
The first result that we prove concerns the bound of the sequence of approximating solutions $\{\un\}$ in $\w12 \cap L^{\infty}(\Omega)$. Let us recall that
\[
\varphi(t)=\int_0^{t}\frac{ds}{(|h(s)|+1)^2}, \quad H(t)=\int_0^t\frac{ds}{|h(s)|+1}.
\]
We notice that $\varphi(u_n)$ belongs to $\w12$, since the real valued function $\varphi(t)$ is Lipschitz. Moreover,  if   the assumption \eqref{ipoh} holds, there exists $C_1>0$ such that $|\varphi(s)| \leq C_1$ for any  $n \in \mathbb{N}$.

\begin{lemma} \label{decay}
Assume \eqref{alfa}, \eqref{ipoh}, and take $\mu\ge0$, $E \in [\elle 2]^N$, and $f \in \elle 1$. Then there exists $C>0$ such that
\begin{equation}
\label{stima_decay}
|\{|\un|>|k|\}|^{\frac{2}{2^*}}\le \frac{C}{|H(k)|^2}\io \big(|E|^2 + |f|\big), \quad  \forall \  n \in \mathbb{N},\, k\in \mathbb{R}
\end{equation}
\end{lemma}
\begin{proof}
Consider the function $v=G_{\varphi(k)}(\varphi(\un^+))\ge 0$ (where $k\ge0$ and $\un^+$ is the positive part of $\un$) and notice that by construction
\be\label{setset}
\begin{split}
\left\{x \in \Omega \, :  v > 0 \right\}=&\left\{x \in \Omega \, : \varphi(\un) > \varphi(k)  \right\}=\left\{x \in \Omega \, :  \un >k \right\}\\
=&\left\{x \in \Omega \, :   H(\un) > H(k)  \right\}
\end{split}
\ee
Taking $v$ as a test function in \eqref{approx}, dropping the positive zeroth order term, and using the ellipticity of the matrix $M(x)$, we get
\[
\begin{split}
\alpha\int_{\{k<\un^+\}} \frac{|\nabla \un^+|^2}{(|h(\un^+)|+1)^2}\le & \int_{\{k<\un^+\}} |h(\un^+)||E| \frac{|\nabla \un^+|}{(|h(\un^+)|+1)^2}+\int_{\{k<\un^+\}} |f| \varphi(\un^+)\\
\le &\frac{1}{2\alpha}\int_{\{k<\un^+\}} |E(x)|^2+\frac\alpha 2\int_{\{k<\un^+\}} \frac{|\nabla \un^+|^2}{(|h(\un^+)|+1)^2}+C_1 \int_{\{k<\un^+\}}|f(x)|,
\end{split}
\] 
where we have used Young's inequality and that $|\varphi(\un^+)|\le C_1$. Therefore, recalling property \eqref{setset} and that $H'(s)=\frac{1}{(|h(s)|+1)^2}$, it follows that
\be\label{stampac}
\begin{split}
\frac{\alpha}{2\mathcal{S}^2}\left(\io | G_{H(k)}(H(\un^+))|^{2^*}\right)^{\frac2{2^*}}\le & \frac{\alpha}{2}\int_{\{k<\un^+\}} |\nabla H(\un^+)|^2\\ \le&  \frac{1}{2\alpha}\int_{\{k<\un^+\}} |E(x)|^2+ C_1 \int_{\{k<\un^+\}}|f(x)|.
\end{split}
\ee
Taking $k=0$, we deduce that for any $t>0$
\[
\begin{split}
|H(t)|^2|\{\un^+> t\}|^{\frac{2}{2^*}}\le& \left(\int_{\{\un^+\ge t\}} | H(\un^+)|^{2^*}\right)^{\frac2{2^*}}\\ \le&\left(\io | H(\un^+)|^{2^*}\right)^{\frac2{2^*}}\le C \io \left(|E(x)|^2+ |f(x)|\right),
\end{split}
\]
where we have used Chebyshev's inequality.
Repeating the same argument with $k\ge0$ and $v=G_{\varphi(k)}(\varphi(\un^-))\le 0$ (we use the notation $u=u^++u^-$), one gets
\be\label{stampacbis}
\frac{\alpha}{2\mathcal{S}^2}\left(\io | G_{H(k)}(H(\un^-))|^{2^*}\right)^{\frac2{2^*}}  \le  \frac{1}{2\alpha}\int_{\{\un^-<-k\}} |E(x)|^2+ C_1 \int_{\{\un^-<-k\}}|f(x)|,
\ee
and, for any $t>0$,
\[
|H(-t)|^2|\{\un^-< -t\}|^{\frac{2}{2^*}} \le C \io \left(|E(x)|^2+ |f(x)|\right).
\]
Putting together the estimates for $\un^+$ and $\un^-$ we obtain the desired result.
\end{proof}

%
%
%

Let us provide now the proof of existence of bounded solutions.
\begin{proof}[Proof of Theorem \ref{generalh}]
In order to prove the bound in $\elle{\infty}$ for the sequence $\{\un\}$, let us set $\eta=H(k)$ and $z_n=H(\un^+)$. Then, estimate \eqref{stampac} becomes
\[
\left(\io | G_{\eta}(z_n)|^{2^*}\right)^{\frac2{2^*}}\le C \int_{\eta\le|z_n|} \left(|E(x)|^2+ |f(x)|\right).
\]
Since $|E|^2+|f|$ belongs to $\elle m$ with $m>\frac{N}{2}$, it follows by classical results (see \cite{stamp}) that 
\[
\| z_n\|_{\elle{\infty}}=\|H(\un^+)\|_{\elle{\infty}}\le C,
\]
for some $C>0$ independent on $n$.
Taking advantage of assumption \eqref{growthg}, we deduce that 
\[
\|\un\|_{\elle{\infty}}\le H^{-1}(C).
\] Similarly, estimate \eqref{stampacbis} implies that $ \|\un^-\|_{\elle{\infty}}\le C'$, so that $\{\un\}$ is bounded in $\elle{\infty}$.\\ As a consequence, choosing $v=\un$ in \eqref{approx}, it easily follows that $\{\un\}$ is bounded also in $ \w12$.\\
Therefore, we can apply Lemma \ref{passtothelim} to conclude that there exists $u\in\w12\cap\elle{\infty}$ solution to \eqref{problem}. The uniqueness follows from Theorem \ref{comparison}.
\end{proof}

\begin{remark}
Notice that the estimate \eqref{stima_decay} is true for all nonlinearity $h(u)$. However, as we have just seen in the proof of Theorem \ref{generalh}, it is useful only if $H(k)\to\infty$ as $k$ diverges. Another consequence of \eqref{growthg} is that for any $\epsilon >0$ there exists $k_{\epsilon}$ such that $\forall \, k > k_{\epsilon}$ it follows that
\begin{equation}
\label{levelset}
|\{x \in \Omega : \quad |\un(x)|>k\}| \leq \epsilon , \quad \text{uniformly \, w.r. to } n \in \mathbb{N}.
\end{equation}
This estimate is crucial in order to prove the boundedness of $\{\un\}$ in the energy space, at least for some choices of the nonlinearity $h(u)$.
\end{remark}
Les us provide now our first result in the special case $h(s)=s\log(e+|s|)$.
\begin{proof}[Proof of Theorem \ref{logarithmich}]
The first step is to show that, for any $a\ge1$, there exists\\ $C=C(a, f, E)$ such that 
\begin{equation}
\label{stimaloga}
\|\log^a(e+|\un|)\|_{\elle{2^*}}\le C.
\end{equation}
Let us set
\[
v=\int_0^{\un}\frac{\log^{2(a-1)}(e+|s|)}{(e+|s|)^2}ds,
\] 
and notice that $|v|\le C_a$, for some positive constant $C_a$.
Taking $v$ as a test function in \eqref{approx}
and using Young's inequality, we get
\[
\begin{split}
\alpha\io|\nabla \un|^2\frac{\log^{2(a-1)}(e+|\un|)}{(e+|\un|)^2} \le& \io |E(x)||\nabla \un| \frac{\log^{2a-1}(e+|\un|)}{(e+|\un|)^2} +C_a\|f\|_{\elle 1}
\\ \le& \frac{1}{2\alpha}\io\log^{2a}(e+|\un|)|E(x)|^{2}+\frac{\alpha}{2}\io|\nabla \un|^2\frac{\log^{2(a-1)}(e+|\un|)}{(e+|\un|)^2}\\+&C_a\io|f|,\\
\end{split}
\]
that becomes
\[
\begin{split}
\frac{\alpha}{2a^2}\io|\nabla \log^a(e+|\un|)|^2&=\frac{\alpha}{2}\io|\nabla \un|^2\frac{\log^{2(a-1)}(e+|\un|)}{(e+|\un|)^2} 
\\ &\le \frac{1}{2\alpha}\io\log^{2a}(e+|\un|)|E(x)|^{2}+C_a\|f\|_{\elle 1}\\
&\le \frac{1}{\alpha}\int_{\{|\un|\ge k|\}}(\log^{a}(e+|\un|)-1)^2|E(x)|^{2}\\
&+\frac{\log^{2a}(e+k)+2}{2\alpha}\|E\|_{\elle 2}^2+C_a\|f\|_{\elle 1}.
\end{split}
\]
Applying Sobolev's embedding to the function $\log^a(e+|\un|)-1\in W^{1,2}_0(\Omega)$ and thanks to H\"older's inequality, it follows that
\[
\begin{split}
\mathcal{S}^{-2}\|(\log^a(e+|\un|)-1)\|_{\elle{2^*}}^{2}&\le \frac{a^2}{\alpha^2}\left(\int_{\{|\un|\ge k|\}}|E|^{N}\right)^{\frac2N}\|(\log^a(e+|\un|)-1)\|_{\elle{2^*}}^{2}\\
&+C_{k,a}(\|E\|_{\elle 2}^2+\|f\|_{\elle 1}).
\end{split}
\]
Thanks to estimate \eqref{levelset}, the fact that $E\in [\elle{N}]^N$, and the uniform continuity of the integral, there exists $k$ large enough such  that 
\[
\frac{a^2}{\alpha^2}\left(\int_{\{|\un|\ge k|\}}|E|^{N}\right)^{\frac2N}\le \frac1{2\mathcal{S}^2} \quad \forall \ n\in\mathbb{N}.
\]
Therefore, we conclude that 
\[
 \frac1{2\mathcal{S}^2}\|(\log^a(e+|\un|)-1)\|_{\elle {2^*} }^{2} \le C_{k,a} (\|E\|_{\elle 2}^2+\|f\|_{\elle 1}),
\]
and \eqref{stimaloga} follows.

Next, we will prove that $\left\{\un\right\}$ is bounded in $\w12$.
Taking $\un$ as a test function in \eqref{approx} and using Young's inequality, we get
\begin{equation}
\label{notte}
\begin{split}
\frac{\alpha}{2}\io|\nabla \un|^2\le& \frac{1}{2\alpha}\int_{\{|\un|> k|\}}|\un|^2\log^2(e+|\un|)|E|^2\\
+&\frac{k^2\log^2(e+k)}{2\alpha}\|E\|^2_{\elle2}+\|f\|_{\elle{2_{*}}}\|\un\|_{\elle{2^{*}}}.
\end{split}
\end{equation}
To deal with the first integral in the right hand side, notice that 
\be\label{12:29}
\begin{split}
\int_{\{|\un|> k|\}}|\un|^2\log^2(e+|\un|)|E|^2 \le &\left(\int_{\{|\un|> k|\}}|E|^{r}\right)^{\frac2r}\|\un\|_{\elle{2^{*}}}^2\|\log(e+|\un|)\|_{\elle b}^2\\
\le & C_{1} \left(\int_{\{|\un|> k|\}}|E|^{r}\right)^{\frac2r}\|\un\|_{\elle{2^{*}}}^2\\
\le & C_2 \left(\int_{\{|\un|> k|\}}|E|^{r}\right)^{\frac2r}\|\nabla \un\|_{\elle{2}}^2,
\end{split}
\ee
where $\frac1b=\frac1N-\frac1r$, we have used \eqref{stimaloga}, with $a=\max\{{b}/{2^*},1\}$, and Sobolev's inequality.
At this point, estimate \eqref{levelset} again assures us that there exists $k$ large enough so that
$$ C_2  \left(\int_{\{|\un|> k|\}}|E|^{r}\right)^{\frac2r} \leq  \frac{\alpha^2}{2}\quad \forall n\in\mathbb{N}.$$ 
Going back to estimate \eqref{notte}, we obtain
$$
 \frac{\alpha}{4} \|\nabla\un\|_{\elle{2}}^2  \leq
\frac{ k^2\log^2(e+k)}{2\alpha}\|E\|^2_{\elle2}+\|f\|_{\elle{2_{*}}}\|\un\|_{\elle{2^{*}}},
$$
which gives the boundedness of $\{\un\}$ in $\w12$. Since \eqref{12:29}. with $k=0$, implies
\[
\io |\un|^2\log^2(e+|\un|)|E|^2\le C_2 \left(\io|E|^{r}\right)^{\frac2r}\|\nabla \un\|_{\elle{2}}^2,
\] 
we can use Lemma \ref{passtothelim} to conclude that there exists a solution to \eqref{logproblem}.

As far as uniqueness is concerned, notice that in the case under consideration $h(s)=s\log(e+|s|)$ and, therefore,  $|h'(s)|\le 2\ +|s|$. Under our assumption on $E$, for any solution $u\in \w12$, we have that $|u||E|\in \elle 2$. Therefore, Theorem \ref{comparison} assures uniqueness. 

\end{proof}

\subsection{Strongly superlinear case} \label{fixpoint}
Let us start this section with the following auxiliary Lemma.
\begin{lemma}\label{trucco}
For $\delta,\theta>0$ set $R= (\delta(\theta+1))^{-1/\theta}$ and assume that
\[
K\le K_{\delta}=\left(\frac{1}{\delta(\theta+1)}\right)^{\frac{1}{\theta}}\frac{\theta}{\theta+1}.
\]
Then, if $s\in (0,R)$, it follows that
\[
\delta s^{1+\theta}+K\le R.
\]
\end{lemma}
\begin{proof}
For $s\in(0,R)$ and $K\le K_{\delta}$ we have that $\delta s^{1+\theta}+K\le \delta R^{1+\theta}+K_{\delta}$. Using the definition of $R$ and $K_{\delta}$ it follows that
\[
\begin{split}
\delta R^{1+\theta}+K_{\delta}=&\delta \left(\frac{1}{\delta(\theta+1)}\right)^{1+\frac1{\theta}}+\left(\frac{1}{\delta(\theta+1)}\right)^{\frac{1}{\theta}}\frac{\theta}{\theta+1}\\
=&\left(\frac{1}{\delta(\theta+1)}\right)^{\frac{1}{\theta}}\left(\frac{1}{\theta+1}+\frac{\theta}{\theta+1}\right)=R
\end{split}
\]
\end{proof}
Let us give the proof of the main result of this section.
\begin{proof}[Proof of Theorem \ref{fixedpoint}] Assume at first that $u_1,u_2\in \elle{m^{**}}\cap \w12$ are two solutions to \eqref{problembis}. Since, for $i=1,2$, we have that
\[
\io|u_i|^{2\theta}|E|^2\le \|u_i\|_{\elle{m^{**}}}^{2\theta}\|E\|_{\elle r}|\Omega|^{1-\frac2r-\frac{2\theta},{m^{**}}}<\infty,
\]
Theorem \ref{comparison} implies that $u_1$ and $u_2$ must coincide, and the uniqueness in the considered class of solutions is proved.\\
 Let us now address the proof of existence and let us divide it in two steps.\\
\textbf{Step 1.} Take $v\in \elle{m^{**}}$ and consider the following problem
\begin{equation}\label{phuc}
\begin{cases}
-\div(M(x)\nabla w)= -\div(v|v|^{\theta}E(x))+f(x) \qquad & \mbox{in } \Omega,\\
w (x) = 0 & \mbox{on }  \partial \Omega.
\end{cases}
\end{equation}
Since $m\ge 2_*$, the right hand side above belongs to $W^{-1,2}(\Omega)$ and the Lax-Milgram Theorem implies existence and uniqueness of a solution $w\in W^{1,2}_0(\Omega)$. The aim of this first step is to obtain precise estimates of the $\elle{m^{**}}$ norm of suitable truncations of $w$ in term of the norms of $v$ and $f$. For it, set $w_{k,l}=T_k(G_l(w))$ and take $\frac{1}{2\gamma-1}|w_{k,l}|^{2\gamma-2}w_{k,l}$, with $\gamma=\frac{m^{**}}{2^*}$, as a test function in the weak formulation of \eqref{phuc}. Recalling that $2\gamma-2>0$, we get
\be\label{30-10}
\alpha\io|\nabla w_{k,l}|^{2}|w_{k,l}|^{2\gamma-2}\le \io |E(x)||v|^{1+\theta}|\nabla w_{k,l}||w_{k,l}|^{2\gamma-2}+\frac{1}{2\gamma-1}\io|f(x)||w_{k,l}|^{2\gamma-1}.
\ee
In order to deal with the first integral in the right hand side of \eqref{30-10}, we recall that assumption \eqref{27-10bis} implies
\[
1=\frac1r+\frac{1+\theta}{m^{**}}+\frac12+\frac{\gamma-1}{m^{**}}.
\]
Using H\"older's inequality, we obtain that
\[
\begin{split}
\io |E(x)||v|^{1+\theta}|\nabla w_{k,l}||w_{k,l}|&^{2\gamma-2}  \\ \le C_{E,l}&\|v\|_{\elle{m^{**}}}^{1+\theta}\||\nabla w_{k,l}||w_{k,l}|^{\gamma-1}\|_{\elle 2}\|w_{k,l}\|_{\elle{m^{**}}}^{\gamma-1},
\end{split}
\]
where $C_{E,l}=\left(\int_{\{|w|> l\}}|E|^r\right)^{\frac1r}$. Moreover, recalling that $(2\gamma-1)m'=m^{**}=2^*\gamma$, it follows that
\[
\io|f||w_{k,l}|^{2\gamma-1}\le C_{f,l} \|w_{k,l}\|_{\elle{m^{**}}}^{\gamma-1} \left(\io|w_{k,l}|^{2^*\gamma}\right)^{\frac{1}{2^*}}
\]
\[
\le \mathcal{S}\gamma C_{f,l} \|w_{k,l}\|_{\elle{m^{**}}}^{\gamma-1} \left(\io|\nabla w_{k,l}|^2|w_{k,l}|^{2\gamma-2}\right)^{\frac12},
\]
where $C_{f,l}=\left(\int_{\{|w|\ge l\}}|f|^m\right)^{\frac1m}$ and we used Sobolev's inequality in the last line. Plugging these two piece of information in \eqref{30-10}, we get
\[
\frac{\alpha}{\mathcal{S}\gamma}	\left(\io|w_{k,l}|^{m^{**}}\right)^{\frac1{2^*}}\le \alpha\left(\io|\nabla w_{k,l}|^2|w_{k,l}|^{2\gamma-2}\right)^{\frac12}
\]
\[
\le \left[ C_{E,l}\|v\|_{\elle{m^{**}}}^{1+\theta}+\frac{ \mathcal{S}\gamma}{2\gamma-1}C_{f,l}\right]\left(\io|w_{k,l}|^{m^{**}}\right)^{\frac{\gamma-1}{m^{**}}}.
\]
That is
\be\label{18:13}
\|w_{k,l}\|_{\elle{m^{**}}}\le\frac{\mathcal{S}\gamma}{\alpha} \left[ C_{E,l}\|v\|_{\elle{m^{**}}}^{1+\theta}+\frac{ \mathcal{S}\gamma}{2\gamma-1}C_{f,l}\right].
\ee
Taking the limit as $k\to\infty$, $l=0$, and recalling that $\gamma\ge1$, we conclude that
\be\label{15:27}
\|w\|_{\elle{m^{**}}}\le \frac{\mathcal{S}m^{**}}{\alpha2^{*}}  \|E\|_{\elle{r}}\|v\|_{\elle{m^{**}}}^{1+\theta}+\frac{\mathcal{S}^2m^{**}}{\alpha2^{*}}\|f\|_{\elle m}.
\ee
\textbf{Step 2.} In this second step we shall use Schauder's fixed point theorem to prove existence of a solution to \eqref{problembis}. Using Lemma \ref{trucco} with 
\[
\delta=\frac{\mathcal{S}m^{**}}{\alpha2^{*}}   \|E\|_{\elle{r}} \ \ \ \mbox{and} \ \ \ K=\frac{ \mathcal{S}^2m^{**}}{\alpha2^*}\|f\|_{\elle m},
\]
recalling assumption \eqref{smallass}, and thanks to estimate \eqref{15:27}, we conclude that
\[
\|v\|_{\elle{m^{**}}}\le R \ \ \ \Longrightarrow \ \ \ \|w\|_{\elle{m^{**}}}\le \delta \|v\|_{\elle{m^{**}}} + K_{\delta} \le R,
\]
where $R= (\delta(\theta+1))^{-1/\theta}$.
 Therefore, considering the map $v\to S(v)=w$ that associates to any $v\in\elle{m^{**}}$ the unique solution $w\in \elle{m^{**}}$ of \eqref{phuc}, we deduce that $S(B_R)\subset B_R$, namely, the ball of radius $R$ is invariant under the action of $S$.\\
Take now a sequence $\{v_n\}\subset \elle{m^{**}}$ that strongly converges to $v$ in $\elle{m^{**}}$. If $w_n=S(v_n)$ and $w=S(v)$, the function $z_n=w_n-w$ solves
\[
\begin{cases}
-\div(M(x)\nabla z_n)= -\div\big((v_n|v_n|^{\theta}-v|v|^{\theta})E(x)\big) \qquad & \mbox{in } \Omega,\\
z_n (x) = 0 & \mbox{on }  \partial \Omega.
\end{cases}
\]
Following the same procedure used in Step 1 to obtain estimate \eqref{15:27}, we deduce that
\[
\|z_n\|_{\elle{m^{**}}}\le\frac{\mathcal{S}\gamma}{\alpha} \|E\|_{\elle{r}}\|v_n|v_n|^{\theta}-v|v|^{\theta}\|_{\elle{\frac{m^{**}}{1+\theta}}}.
\]
Thanks to the strong convergence of $v_n$ in $\elle{m^{**}}$, the dominated convergence theorem implies that the right hand side above goes to zero as $n\to\infty$. Therefore, the sequence $\{w_n\}$ converges to $w$ in $\elle{m^{**}}$, namely, the map $v\to S(v)$ is continuous.\\
To conclude, we need to show that the map is also compact, namely, if $v_n\rightharpoonup v$ in $\elle{m^{**}}$, up to a subsequence $w_n\to w$ in $\elle{m^{**}}$. Taking $w_n=S(v_n)$ as a test function in the weak formulation of the problem solved by $w_n$, we deduce that
\[
\begin{split}
\alpha\io |\nabla w_n|^2\le& \io |v_n|^{1+\theta}|E||\nabla w_n|+\io |f||w|\\
\le & C(\Omega) \|v_n\|^{\theta+1}_{\elle{m^{**}}}\|E\|_{\elle r}\|\nabla w_n\|_{\elle 2}+\mathcal{S}\|f\|_{\elle{2_*}}\|\nabla w\|_{\elle 2}
\end{split}
\]
where we have used that (recall assumption \eqref{27-10bis} and that $2\le m^*$)
\[
\frac{1+\theta}{m^{**}}+\frac1r+\frac12= \frac{1}{m^*}+\frac12\le 1.
\]
Therefore there exists $\zeta\in\w12$ such that, up to a subsequence, $\wn \to \zeta$ in $\elle 2$ and $a.e.$ in $\Omega$. However, the weak convergence of $v_n$ in $\elle{m^{**}}$ is enough to deduce that $\zeta$ coincides with $w=S(v)$.
Let us now use estimates \eqref{18:13} with $v_n$. Since $\|v_n\|_{\elle{m^{**}}}$ is bounded, we can take the limit as $k\to \infty$ to obtain that for any $l\ge 0$
\[
\|G_l(w_n)\|_{\elle{m^{**}}}\le C\left[\left(\int_{\{|w|> l\}}|E|^r\right)^{\frac1r}+ \left(\int_{|w_n|\ge l}|f|^m\right)^{\frac1m}\right].
\]
We stress that the right hand side above converges to zero as $l\to \infty$, uniformly with respect to $n$, due to the strong convergence of $\{w_n\}$ in $\elle 2$. Now, for any measurable set $E\subset\Omega$, we have that
\[
\int_{E}|w_n|^{m^{**}}\le \int_{E}|T_l(w_n)|^{m^{**}}+\io |G_l(w_n)|^{m^{**}}\le l^{m^{**}}|E|+\omega(l),
\]
with $\omega(l)\to 0$ as $l\to\infty$. Therefore, for any $\epsilon>0$, there exists $l$ large enough such that $\omega(l)\le \frac{\epsilon}2$. Setting $\delta= \frac{\epsilon}{2l^{m^{**}}}$ we have that
\[
\int_{E}|w_n|^{m^{**}}\le \epsilon \ \ \ \forall E \ : \ |E|<\delta.
\]
This means that the sequence is equi-integrable in $\elle{m^{**}}$ and Vitali Theorem assures us the required strong convergence.
\end{proof}

%

We prove now our result in the symmetric setting. 
\begin{proof}[Proof of Proposition \ref{radialcase}]
Let us start assuming that \eqref{onedproblem} has a solution $u\in \elle{\infty}\cap W^{1,2}_0(\Omega)$.  It follows that $u$ solves
\[
-\Delta u= K h'(u)\nabla u\frac{x}{|x|}+h(u)\frac{K(N-1)}{|x|}+\frac{\epsilon}{N-1}\frac1{|x|}
\]
in the weak sense and, therefore, we can apply Theorems 8.8 and 8.12 of \cite{book} (with $c=K(N-1) h'(u)x/|x|$ and $f=h(u){K}/{|x|}+{\epsilon}/{|x|(N-1)}$) to deduce that $u\in W^{2,2}(\Omega)$ and that the equation above is also satisfied a.e. in $\Omega$.
Now, since Theorem \ref{comparison} implies uniqueness, $u$ has to be radial and the equation above becomes
\[
-\frac{1}{r^{N-1}}\frac{d}{dr}\left(r^{N-1}\frac{d}{dr}\bar u\right)= K  \frac{d}{dr} h(\bar u)+h(\bar u)\frac{K(N-1)}{r}+\frac{\epsilon}{N-1}\frac1{r}.
\]
Multiplying both side by $r^{N-1}$ and integrating between $0$ and $r$, 
we get 
\[
-\bar u'(r)=K h(\bar u)+{\epsilon}.
\]
By separation of variables, it follows
\[
\int_{\bar u(R)=0}^{\bar u(0)}\frac{ds}{Kh(s)+{\epsilon}}=R
\]
that clearly implies \eqref{oss}.\\

On the other hand, assume that \eqref{oss} holds true and let $a>0$ such that 
\[
\int_{0}^{a}\frac{ds}{Kh(s)+{\epsilon}}=R
\]
If we set $\bar u$ to be the inverse function of $t\to \int_{t}^{a}\frac{ds}{h(s)+{\epsilon}}$, it follows, doing the previous computation backwards, that $u(x)=\bar u(|x|)$ is the unique solution to \eqref{onedproblem}.
\end{proof}
\begin{remark}[Reaction non coercive term]\label{explicit1}
If we take $h(s)=s^2$, the ODE
\[
-u'(r)=-K u^{2}+{\epsilon},  \quad \mbox{and} \quad \ u(R)=0
\]
can be solved by separation of variable 
\[
\int_{a}^{u(r)}\frac{ds}{Ks^{2}+{\epsilon}}=-r,
\]
that explicitly provides
\[
\frac{1}{\sqrt{K{\epsilon}}}\left[\arctan{\left(u(r)\sqrt{\frac{K}{{\epsilon}}}\right)}-\arctan{\left(a\sqrt{\frac{K}{{\epsilon}}}\right)}\right]=-r.
\]
Imposing that $u(R)=0$ we find that
\[
a=\sqrt{\frac{{\epsilon}}{K}}\tan{\left(\sqrt{K{\epsilon}}R\right)}.
\]
Since our solution has to be bounded and positive, we deduce that $\sqrt{K{\epsilon}}R<\frac{\pi}{2}$. Eventually we get
\[
u(r)=\sqrt{\frac{{\epsilon}}{K}}\tan{\left(\sqrt{K{\epsilon}}(R-r)\right)}.
\]
\end{remark}
\begin{remark}[Absorption coercive term]\label{explicit2}
It is also interesting to consider the following equation
\[
-u'(r)=-K u^{2}+{\epsilon},  \quad \mbox{and} \quad \ u(R)=0,
\]
namely assuming \emph{the good sign condition} for the divergence of $E$. In this case one expects that the nonlinear term should not affect existence of solution. Indeed one gets
\[
\int_{a}^{u(r)}\frac{ds}{Ku^{2}-{\epsilon}}=r,
\]
that is
\[
-\frac{1}{\sqrt{K{\epsilon}}}\left[\arctanh{\left(u(r)\sqrt{\frac{K}{{\epsilon}}}\right)}-\arctanh{\left(a\sqrt{\frac{K}{{\epsilon}}}\right)}\right]=r.
\]
We recall that $\arctanh$ is a function defined in $(-1,1)$ strictly increasing, with $\arctanh(0)=0$, and with image $(-\infty,+\infty)$. Then for any values of  $K,{\epsilon}, R$ we find  
\[
0<a=\sqrt{\frac{{\epsilon}}{K}}\tanh\left(\sqrt{K{\epsilon}}R\right).
\]
The solution is
\[
u(r)=\sqrt{\frac{K}{{\epsilon}}}\tanh{\left(\sqrt{K{\epsilon}}(R-r)\right)}.
\]
\end{remark}

\subsection{The impact of the zero order term}
In this section we deal with problem \eqref{problem} with $\mu>0$. We take advantage of an $L^1$ a priori estimate due to the presence of the zeroth lower order term.

\begin{lemma}\label{elle1estimate}
Assume \eqref{alfa} and the first two conditions in \eqref{ipoh}. Take $E\in[\elle 2]^N$,  $f\in\elle 1$, $\mu >0$, and let $\un$ be the weak solution of the problem \eqref{problem}. Therefore, we have
\begin{equation}
\label{52}
\io |\un|\le \frac{1}{\mu}\io |f|.
\end{equation}
\end{lemma}
\begin{proof}
Let us take $T_{\epsilon}(u_n)$ as a test function in \eqref{approx} to obtain
\[
 \alpha\io |\nabla T_{\epsilon}(u_n)|^2+\mu\io u_nT_{\epsilon}u_n)\le h(\epsilon)\io|E(x)||\nabla T_{\epsilon}(u_n) |+ \epsilon \|f\|_{\elle 1}.
\]
Dividing everything by $\epsilon$, using Young's inequality, and recalling that $h(0)=0$ and that $h\in W^{1,\infty}(\mathbb{R})$, we get
\[
\begin{split}
\frac{\mu}{\epsilon}\io u_nT_{\epsilon}(u_n)\le \frac{\alpha}{2\epsilon}\io|\nabla T_{\epsilon}(u_n)|^2+\mu\io u_nT_{\epsilon}(u_n)\le L\frac{\epsilon}{2\alpha}\io|E(x)|^2+\|f\|_{\elle 1},
\end{split}
\]
where $L$ is the Lipschitz constant of $h(s)$ for $s\in(-1,1)$. Taking the limit as $\epsilon \to 0$, we finally obtain \eqref{52}.
\end{proof}

We can now prove Theorem \ref{+uteo}

\begin{proof}[Proof of Theorem \ref{+uteo}]
Thanks to estimate \eqref{52}, we can fix $k>0$, independent on $n$, such that
\[
2^{1+\theta}\mathcal{S}\mu^{-1}\|f\|_{\elle 1}^{\theta}   \left(\int_{|u_n|\ge k} |E(x)|^r\right)^{\frac1r}\le \frac{\alpha}{4}.
\]
Taking $G_k(u_n)$ as a test function in \eqref{approx}, we get
\[
\begin{split}
\alpha\io|\nabla G_k(u_n)|^2\le& 2^{1+\theta}\io |G_k(u_n)|^{1+\theta} |E(x)||\nabla G_k(u_n)|\\
+& (2k)^{1+\theta}\io|E(x)||\nabla G_k(u_n)|+\io |f||G_{{k}}(u_n)|,
\end{split}
\]
where we have used that $u_n=T_k(u_n)+G_k(u_n)$ and that $|a+b|^{1+\theta}\le 2^{1+\theta} ( |a|^{1+\theta}+|b|^{1+\theta})$.
To estimate the first term in the right hand side above, we take advantage of the fact that the integral is carried over super level sets of the functions $u_n$. More specifically,  by assumption on $\theta$,  using the H\"older's inequality  and Lemma \eqref{elle1estimate} we have that
\[
\begin{split}
 2^{1+\theta} \io |G_k(u_n)|^{1+\theta} |E(x)||\nabla G_k(u_n)|=2^{1+\theta} \int_{|u_n|\ge k} |G_k(u_n)| |G_k(u_n)|^{\theta}|E(x)||\nabla G_k(u_n)|&\\ \le  2^{1+\theta}\|G_k(u_n)\|_{\elle{2^*}}\|G_k(u_n)\|_{\elle 1}^{\theta}\left(\int_{|u_n|\ge k} |E(x)|^r\right)^{\frac1r}\|\nabla G_k(u_n)\|_{\elle 2}&\\
\le 2^{1+\theta}\mathcal{S}\mu^{-1}\|f\|_{\elle 1}^{\theta}   \left(\int_{|u_n|\ge k} |E(x)|^r\right)^{\frac1r}   \|\nabla G_k(u_n)\|_{\elle 2}^2&.
\end{split}
\]
The choice of $ k$ implies that
\[
2^{1+\theta} \io |G_{{k}}(u_n)|^{1+\theta} |E(x)||\nabla G_{{k}}(u_n)|\le \frac{\alpha}{4}\io|\nabla G_{{k}}(u_n)|^2 
\]
On the other hand a simple application of Young's inequality allows us to write
\[
 (2k)^{1+\theta}\io|E(x)||\nabla G_{{k}}(u_n)|+\io |f||G_{{k}}(u_n)|\le C\left(k^{2(1+\theta)}\|E\|_{\elle2}^2+\|f\|_{\elle{2_*}}\right)+\frac{\alpha}{4}\io|\nabla  G_{{k}}(u_n)|^2 .
\]
Therefor we deduce that
\[
\frac{\alpha}{2}\io|\nabla G_{{k}}(u_n)|^2\le C\left(k^{2(1+\theta)}\|E\|_{\elle2}^2+\|f\|_{\elle{2_*}}^2\right).
\]
To conclude the argument, let us take $T_{{k}}(u_n)$ as a test function in \eqref{approx}. It follows that
\[
\frac{\alpha}{2}\io|\nabla T_{{k}}(u_n)|^2\le \frac{k^{2(1+\theta)}}{\alpha 2}\|E\|_{\elle2}^2+{k}\|f\|_{\elle1}.
\]
Putting together the estimate for $\nabla G_k(u_n)$ and $\nabla T_k(u_n)$, we recover that $\|\un\|_{\w12}\le C'$. Moreover, recalling that $r=N/(1-\theta N)$, we have that
\[
\int_{\Omega} |\un|^{2+2\theta}|E|^{2}\le \|\un\|_{\elle{2^*}}^{2}\|u\|_{\elle 1}^{2\theta}\|E\|_{L^{r}(\Sigma)}^{2}.
\]
At last, we can apply Lemma \ref{passtothelim} and Theorem \ref{comparison} to conclude that there exists a unique solution to \eqref{zeroorderterm}.
\end{proof}
\begin{remark}\label{lastrem} Let us show that if $\theta>\frac{1}{N}$ there is no hope to have existence of solution to \eqref{zeroorderterm} for general $f,E$. Assume that $\Omega$ is the ball centered in the origin with radius one, $E=-\frac{x}{|x|}$, $f\ge0$, and that there exists a solution $u\in\w12$ to \eqref{zeroorderterm}. Taking $v=1-|x|^2$ as a test function, we get
\[
\io fv+2 \io u^{\theta+1}|x|= \io u(v-\Delta v)\le C \io u \le \epsilon \io u^{\theta+1}|x|+C(\epsilon)\io|x|^{-\frac1{\theta}}.
\]
Choosing, for instance, $\epsilon=1$ and dropping the terms involving $u^{\theta+1}$, we end up with
\[
\io fv\le C \int_0^1 r^{-\frac1{\theta}+N-1},
\]
with the right hand side being a finite constant thanks to the assumption $\theta>\frac1N$. This is a necessary condition on $f$ in order to have existence of solutions.
\end{remark}

We can now give the proof of our last result.
\begin{proof}[Proof of Theorem \ref{logsummE}]

Taking $\un$ as a test function and using Young's inequality, we get
\be\label{14:33} 
\frac\alpha2\io|\dun|^2
\leq\frac1{2\alpha}\io|\un|^2\,\log^2(e+|\un|) \,|E|
+\io|f||\un|
\ee
To deal with the third integral in the formula above, we split it in two terms as follows
\[
\begin{split}
\disp 
\io&|\un|^2\,\log^2(e+|\un|) \,|E|^2
\\
\disp
\le&
\rho^2\ipr|\un|^2\,\log^2(e+|\un|)
+\int_{|E|>\rho}|\un|^2\,\log^2(e+|\un|) \,|E|^2
\\
\disp
=&I_1 + I_2
\end{split}
\]
Taking $\lambda\in (0,\frac{2^*-2}{2^*-1})$, we get
\[
\begin{split}
I_1\le& \rho^{2}\io|\un|^{2-\lambda}\,|\un|^{\lambda}\log^2(e+|\un|)\\
\le & \rho^{2} \|\un\|_{\elle{2^*}}^{2-\lambda}\left(\io |\un|^{\frac{2^*\lambda}{2^*-2+\lambda}}\log^{\frac{2^*2}{2^*-2+\lambda}}(e+|\un|) \right)^{1-\frac{2-\lambda}{2^*}}\\
\le& \rho^{2} C_1 \|\un\|_{\elle{2^*}}^{2-\lambda}(|\Omega|+\|\un\|_{\elle 1})^{1-\frac{2-\lambda}{2^*}}\le \rho^2 C_2\|\nabla \un\|_{\elle{2}}^{2-\lambda},
\end{split}
\]
where we have used that $\frac{2^*\lambda}{2^*-2+\lambda}<1$ and that $s^{\frac{2^*\lambda}{2^*-2+\lambda}}\log^{\frac{2^*2}{2^*-2+\lambda}}(e+s)\le C(1+s)$ for $s>0$ and some constant $C$ (that depends on $\lambda$).\\
To estimate $I_2$, we take advantage of the inequality
\[
s\,t\leq A s\,\log(e+s)+ A e^{t/A} \quad \forall s,t\ge0,
\]
that can easily proven showing that the the function $f(t)=\lambda s\,\log(1+s)+ \lambda e^{t/\lambda}-st$ is nonnegative.
Setting $s=|E|$ and $\;t=\log(e+|\un|)$ with $A>N$, we deduce that
\be\label{10:28}
 |E|\, \log(e+|u_n|)  \leq  A\,|E| \,\log(e+|E|)+ A (e+|u_n|)^\frac1A.
\ee
We get
\[
\begin{split}
I_2\le& 2A^2\int_{|E|>\rho} |\un|^2 \;|E|^2\, \log^2(1+|E|) 
+2A^2\int_{|E|>\rho}|\un|^2 \,(e+|\un|)^\frac2A\\
\le& 2A^2 \left[\left(\int_{|E|>\rho}|E|^N\log^N(e+|E|) \right)^{\frac2N}
+2A^2 \|1+|\un|\|_{\elle 1}^{2/A}|\{|E|>\rho\}|^{\frac2N-\frac2A}\right]\|\un\|_{\elle{2^*}}^2
\end{split}
\]
Notice that the term in the square brackets goes to zero as $\rho\to\infty$. Therefore, the contribution of $I_2$ in estimate \eqref{14:33} can be absorbed into the left hand side and we obtain that
\[
 \|\nabla\un\|_{\elle{2}}^2\le  C_3 \rho^2 \|\nabla\un\|_{\elle{2}}^{2-\lambda} +\|f\|_{\elle{(2^*)'}}.
\]
This implies that $\{\un\}$ is a bounded subset of $\w12$.\\
We have to show now that the sequence $\{|\un|^2\,\log^2(e+|\un|) \,|E|^2\}$ bounded. To do it we use again \eqref{10:28}, with $A>N$, to get
\[
\begin{split}
\int_{\Omega}|\un|^2\,\log^2(e+|\un|) \,|E|^2\le 2A^2\int_{\Omega} |\un|^2 \;|E|^2\, \log^2(e+A\,|E|) 
+2\int_{\Omega} |\un|^2 \,(e+|\un|)^\frac2A\\
\le 2A^2 \|\un\|_{\elle{2^*}}\left(\int_{\Omega} |E|^N\log^N(e+A\,|E|) \right)^{\frac2N}+2 \|\un\|_{\elle{2^*}}^2\|1+|\un|\|_{\elle 1}^{\frac2A}|\Omega|^{\frac2N-\frac2A}.
\end{split}
\]
At this point we can use Lemma \ref{passtothelim} to conclude that $\un\rightharpoonup u$ in $\w12$ and that $u$ solves \eqref{problemalog}. As far as uniqueness is concerned notice that if $h(s)=s\log(e+|s|)$ then $|h'(s)|\le c+|s|$ and clearly for any for any solution $u\in \w12$ we have that $|u||E|\in \elle 2$. Therefore Theorem \ref{comparison} assures uniqueness. 
\end{proof}

 \section*{Acknowledgement} 
\noindent S.B. is supported by the Austrian Science Fund (FWF) projects F65, P32788 and FW506004, by the Belgian National Fund for Scientific Research (NFSR) project 40006150, and by the GNAMPA-INdAM Project 2023 ``Regolarità per problemi ellittici e parabolici con crescite non standard" (CUP\_E53C22001930001).


\begin{bibdiv}
	

\begin{biblist}

\bib{tom}{article}{
   author={Abdellaoui, B.},
   author={Fern\'{a}ndez, A. J.},
   author={Leonori, T.},
   author={Younes, A.},
   title={Deterministic KPZ-type equations with nonlocal ``gradient terms''},
   journal={Ann. Mat. Pura Appl. (4)},
   volume={202},
   date={2023},
   number={3},
   pages={1451--1468},
}


\bib{porrettacon}{article}{
   AUTHOR = {Achdou, Y.}
 AUTHOR = {Porretta, A.}
     TITLE = {Mean field games with congestion},
   JOURNAL = {Ann. Inst. H. Poincar\'{e} C Anal. Non Lin\'{e}aire},
    VOLUME = {35},
      YEAR = {2018},
    NUMBER = {2},
     PAGES = {443--480},
}

\bib{david}{article}{
   author={Arcoya, D.},
   author={De Coster, C.},
   author={Jeanjean, L.},
   author={Tanaka, K.},
   title={Continuum of solutions for an elliptic problem with critical
   growth in the gradient},
   journal={J. Funct. Anal.},
   volume={268},
   date={2015},
   number={8},
   pages={2298--2335},
}

\bib{betta}{article}{
   author={Betta, M. F.},
   author={Ferone, V.},
   author={Mercaldo, A.},
   title={Regularity for solutions of nonlinear elliptic equations},
   language={English, with English and French summaries},
   journal={Bull. Sci. Math.},
   volume={118},
   date={1994},
   number={6},
   pages={539--567}
}

\bib{betta1}{article}{
   author={Betta, M. F.},
   author={Di Nardo, R.},
   author={Mercaldo, A.},
   author={Perrotta, A.},
   title={Gradient estimates and comparison principle for some nonlinear
   elliptic equations},
   journal={Commun. Pure Appl. Anal.},
   volume={14},
   date={2015},
   number={3},
   pages={897--922},
}


\bib{Bumi2009}{article}{
AUTHOR = {Boccardo, L.},
     TITLE = {Some developments on {D}irichlet problems with discontinuous
              coefficients},
   JOURNAL = {Boll. Unione Mat. Ital. (9)},
    VOLUME = {2},
      YEAR = {2009},
    NUMBER = {1},
     PAGES = {285--297},
}

\bib{Bumi2012}{article}{
   AUTHOR = {Boccardo, L.},
     TITLE = {Finite energy solutions of nonlinear {D}irichlet problems with
              discontinuous coefficients},
   JOURNAL = {Boll. Unione Mat. Ital. (9)},
    VOLUME = {5},
      YEAR = {2012},
    NUMBER = {2},
     PAGES = {357--368}
}

\bib{jde}{article}{
   author={Boccardo, L.},
   title={Dirichlet problems with singular convection terms and
   applications},
   journal={J. Differential Equations},
   volume={258},
   date={2015},
   number={7},
   pages={2290--2314}
}

\bib{BB}{article}{
 AUTHOR = {Boccardo, L.},
AUTHOR = {Buccheri, S.},
     TITLE = {A nonlinear homotopy between two linear {D}irichlet problems},
   JOURNAL = {Rev. Mat. Complut.},
    VOLUME = {34},
      YEAR = {2021},
    NUMBER = {2},
     PAGES = {541--558},
}

\bib{bbc}{article}{
   author={Boccardo, L.},
   author={Buccheri, S.},
   author={Cirmi, G. R.},
   title={Two linear noncoercive Dirichlet problems in duality},
   journal={Milan J. Math.},
   volume={86},
   date={2018},
   number={1},
   pages={97--104}
}

\bib{bop}{article}{
   author={Boccardo, L.},
   author={Orsina, L.},
   author={Porretta, A.},
   title={Some noncoercive parabolic equations with lower order terms in
   divergence form},
   note={Dedicated to Philippe B\'{e}nilan},
   journal={J. Evol. Equ.},
   volume={3},
   date={2003},
   number={3},
   pages={407--418},
}

\bib{bopbis}{article}{
   author={Boccardo, L.},
   author={Orsina, L.},
   author={Porzio, M. M.},
   title={Regularity results and asymptotic behavior for a noncoercive
   parabolic problem},
   journal={J. Evol. Equ.},
   volume={21},
   date={2021},
   number={2},
   pages={2195--2211},
}


\bib{bmp84}{article}{
   author={Boccardo, L.},
   author={Murat, F.},
   author={Puel, J.-P.},
   title={R\'{e}sultats d'existence pour certains probl\`emes elliptiques
   quasilin\'{e}aires},
   journal={Ann. Scuola Norm. Sup. Pisa Cl. Sci. (4)},
   volume={11},
   date={1984},
   number={2},
   pages={213--235}
}


\bib{bmp}{article}{
   author={Boccardo, L.},
   author={Murat, F.},
   author={Puel, J.-P.},
   title={$L^\infty$ estimate for some nonlinear elliptic partial
   differential equations and application to an existence result},
   journal={SIAM J. Math. Anal.},
   volume={23},
   date={1992},
   number={2},
   pages={326--333}
}






\bib{bucc}{article}{
   author={Buccheri, S.},
   title={Gradient estimates for nonlinear elliptic equations with first
   order terms},
   journal={Manuscripta Math.},
   volume={165},
   date={2021},
   number={1-2},
   pages={191--225}
}

\bib{salva}{article}{
   author={Carmona, J.},
   author={L\'{o}pez-Mart\'{\i}nez, S.},
   author={Mart\'{\i}nez-Aparicio, P. J.},
   title={A priori estimates for non-coercive Dirichlet problems with
   subquadratic gradient terms},
   journal={J. Differential Equations},
   volume={366},
   date={2023},
   pages={292--319},
}

\bib{cirmi2020}{article}{
   author={Cirmi, G. R.},
   author={D'Asero, S.},
   author={Leonardi, S.},
   title={Morrey estimates for a class of elliptic equations with drift
   term},
   journal={Adv. Nonlinear Anal.},
   volume={9},
   date={2020},
   number={1},
   pages={1333--1350},
}

\bib{cirmi2022}{article}{
   author={Cirmi, G. R.},
   author={D'Asero, S.},
   author={Leonardi, S.},
   author={Porzio, M. M.},
   title={Local regularity results for solutions of linear elliptic
   equations with drift term},
   journal={Adv. Calc. Var.},
   volume={15},
   date={2022},
   number={1},
   pages={19--32},
}


\bib{cornalba}{article}{
   author={Cornalba, F.},
   author={Shardlow, T.},
   author={Zimmer, J.},
   title={A regularized Dean-Kawasaki model: derivation and analysis},
   journal={SIAM J. Math. Anal.},
   volume={51},
   date={2019},
   number={2},
   pages={1137--1187}
}


\bib{antonio}{article}{
   author={De Coster, C.},
   author={Fern\'{a}ndez, A. J.},
   title={Existence and multiplicity for an elliptic problem with critical
   growth in the gradient and sign-changing coefficients},
   journal={Calc. Var. Partial Differential Equations},
   volume={59},
   date={2020},
   number={3},
   pages={Paper No. 97, 34},
}

\bib{degond}{article}{
   author={Degond, P.},
   author={Goudon, T.},
   author={Poupaud, F.},
   title={Diffusion limit for nonhomogeneous and non-micro-reversible
   processes},
   journal={Indiana Univ. Math. J.},
   volume={49},
   date={2000},
   number={3},
   pages={1175--1198},
}

\bib{delvecchio}{article}{
 AUTHOR = {Del Vecchio, T.},
 AUTHOR = {Posteraro, M. R.},
     TITLE = {An existence result for nonlinear and noncoercive problems},
   JOURNAL = {Nonlinear Anal.},
    VOLUME = {31},
      YEAR = {1998},
    NUMBER = {1-2},
     PAGES = {191--206}
}

\bib{farroni2021}{article}{
   author={Farroni, F.},
   author={Greco, L.},
   author={Moscariello, G.},
   author={Zecca, G.},
   title={Noncoercive quasilinear elliptic operators with singular lower
   order terms},
   journal={Calc. Var. Partial Differential Equations},
   volume={60},
   date={2021},
   number={3},
   pages={Paper No. 83, 20}
}

\bib{farronibis}{article}{
   author={Farroni, F.},
   author={Greco, L.},
   author={Moscariello, G.},
   author={Zecca, G.},
   title={Nonlinear evolution problems with singular coefficients in the
   lower order terms},
   journal={NoDEA Nonlinear Differential Equations Appl.},
   volume={28},
   date={2021},
   number={4},
   pages={Paper No. 38, 25},
}

\bib{feng}{article}{
   author={Feng, Z.},
   author={Zhang, A.},
   author={Gao, H.},
   title={Local boundedness under nonstandard growth conditions},
   journal={J. Math. Anal. Appl.},
   volume={526},
   date={2023},
   number={2},
   pages={Paper No. 127280, 19},
}

\bib{book}{book}{
   author={Gilbarg, D.},
   author={Trudinger, N. S.},
   title={Elliptic partial differential equations of second order},
   series={Classics in Mathematics},
   note={Reprint of the 1998 edition},
   publisher={Springer-Verlag, Berlin},
   date={2001},
}


\bib{goffi}{article}{
 AUTHOR = {Goffi, A.},
     TITLE = {On the optimal $L^q$-regularity for viscous Hamilton–Jacobi equations with subquadratic growth in the gradient},
   JOURNAL = {	Commun. Contemp. Math},
      YEAR = {2023},
}


\bib{gomezcastro}{article}{
   author={G\'omez Castro, D.},
   title={Beginner's guide to Aggregation-Diffusion Equations},
   journal={preprint https://arxiv.org/pdf/2309.13713.pdf}
}

\bib{grenon}{article}{
   author={Grenon, N.},
   author={Murat, F.},
   author={Porretta, A.},
   title={A priori estimates and existence for elliptic equations with
   gradient dependent terms},
   journal={Ann. Sc. Norm. Super. Pisa Cl. Sci. (5)},
   volume={13},
   date={2014},
   number={1},
   pages={137--205}
}

\bib{hillen}{article}{
AUTHOR = {Hillen, T.},
author={Painter, K. J.}
     TITLE = {A user's guide to {PDE} models for chemotaxis},
   JOURNAL = {J. Math. Biol.},
    VOLUME = {58},
      YEAR = {2009},
    NUMBER = {1-2},
     PAGES = {183--217}
}
\bib{iuorio}{article}{
AUTHOR = {Iuorio, A.},
AUTHOR = {Jankowiak, G.},
AUTHOR = {Szmolyan, P.},
AUTHOR = {Wolfram, M.T.},
     TITLE = {A {PDE} model for unidirectional flows: stationary profiles
              and asymptotic behaviour},
   JOURNAL = {J. Math. Anal. Appl.},
    VOLUME = {510},
      YEAR = {2022},
    NUMBER = {2},
     PAGES = {Paper No. 126018, 33},
}


\bib{marah}{article}{
   author={Marah, A.},
   author={Redwane, H.},
   title={Existence Result for Solutions to Some Noncoercive Elliptic
   Equations},
   journal={Acta Appl. Math.},
   volume={187},
   date={2023},
   pages={Paper No. 18},
}

\bib{phuc}{article}{
   author={Phuc, N. C.},
   title={Morrey global bounds and quasilinear Riccati type equations below
   the natural exponent},
   language={English, with English and French summaries},
   journal={J. Math. Pures Appl. (9)},
   volume={102},
   date={2014},
   number={1},
   pages={99--123},
}

\bib{porretta}{article}{
AUTHOR = {Porretta, A.},
     TITLE = {Weak solutions to {F}okker-{P}lanck equations and mean field
              games},
   JOURNAL = {Arch. Ration. Mech. Anal.},
    VOLUME = {216},
      YEAR = {2015},
    NUMBER = {1},
     PAGES = {1--62}
}

\bib{sake}{article}{
    AUTHOR = {Sakellaris, G.},
     TITLE = {Scale invariant regularity estimates for second order elliptic
              equations with lower order coefficients in optimal spaces},
   JOURNAL = {J. Math. Pures Appl.},
    VOLUME = {156},
      YEAR = {2021},
     PAGES = {179--214},
}


\bib{stamp}{article}{
   author={Stampacchia, G.},
   title={Le probl\`eme de Dirichlet pour les \'{e}quations elliptiques du
   second ordre \`a coefficients discontinus},
   journal={Ann. Inst. Fourier (Grenoble)},
   volume={15},
   date={1965},
   pages={189--258}
}

\end{biblist}

\end{bibdiv}
\end{document}